\documentclass[12pt]{amsart}
\usepackage{amsmath, amssymb, amsthm,verbatim,xypic,epic,eepic,longtable}

\numberwithin{equation}{section}

\newtheorem{thm}[equation]{Theorem}

\newtheorem{lemma}[equation]{Lemma}
\newtheorem{cor}[equation]{Corollary}

\theoremstyle{remark}
\newtheorem*{remark}{Remark}

\newcommand{\F}{\mathbb{F}}

\renewcommand{\P}{\mathbb{P}}

\begin{document}

\title[Rational points on Fermat curves and surfaces]{Rational points on some Fermat curves and surfaces over finite fields}

\author{Jos\'e Felipe Voloch}
\address{
  Department of Mathematics,
  University of Texas,
  Austin, TX 78712 USA
}
\email{voloch@math.utexas.edu}
\urladdr{http://www.ma.utexas.edu/users/voloch/}

\author{Michael E. Zieve}
\address{
  Department of Mathematics,
  University of Michigan,
  530 Church Street,
  Ann Arbor, MI 48109-1043 USA
}
\email{zieve@umich.edu}
\urladdr{http://www.math.lsa.umich.edu/$\sim$zieve/}



\date{March 9, 2013}

\thanks{The first author was partially supported by the Simons Foundation (grant \#234591). The second author was partially supported by NSF grant DMS-1162181.  The authors thank the referee for spotting a typographical error.}


\begin{abstract}
We give an explicit description of the $\F_{q^i}$-rational points on the Fermat curve $u^{q-1}+v^{q-1}+w^{q-1}=0$,
for $i\in\{1,2,3\}$.  As a consequence, we observe that for any such point $(u,v,w)$, the product $uvw$ is a cube in
$\F_{q^i}$.  We also describe the $\F_{q^2}$-rational points on the Fermat surface $u^{q-1}+v^{q-1}+w^{q-1}+x^{q-1}=0$, and show that the product of the coordinates of any such point is a square.
\end{abstract}

\maketitle


\section{Introduction}

Let $q=p^r$ where $p$ is prime and $r$ is a positive integer.  We will examine the $\F_{q^i}$-rational points on
the Fermat curve $u^{q-1}+v^{q-1}+w^{q-1}=0$, for $i\in\{1,2,3\}$.  Several authors have given estimates
for the number of such points, or even formulas for this number in some cases \cite{DH,GV,HS,Moisio}.
We will go further and explicitly write down all the points.  We find the following unexpected consequence:

\begin{cor} \label{cor}
For $i\in\{1,2,3\}$ and $u,v,w\in\F_{q^i}$, if $u^{q-1}+v^{q-1}+w^{q-1}=0$ then $uvw$ is a cube in\/ $\F_{q^i}$.
\end{cor}

Although the shape of this assertion is suggestive, we do not know any generalizations of this result.  Nonetheless,
we highlight this result in case it might inspire further developments.  In the recent paper \cite{ScherrZ},
this result was shown for even $q$ when $i=3$, in order to prove
a conjecture from \cite{SZ} about certain functions related to finite projective planes.  That proof was quite
computational, and the desire for a more conceptual proof led to the present paper.

We now describe the points on these Fermat curves.  Write  $T(X) := X^{q^2} + X^q + X$.

\begin{thm} \label{3}
The points in\/ $\P^2(\F_{q^3})$ on the curve $u^{q-1} + v^{q-1} + w^{q-1} = 0$ are as follows:
\begin{enumerate}
\item points $(c v^{q+1} : v : 1)$ with $c\in\F_q^*$ and $v$ a nonzero root of $T(X)$;
\item points $(c v^{-q} : v : 1)$ with $c\in\F_q^*$ and $v$ a nonzero root of $T(1/X)$;
\item points $(u : v : w)$ with $u,v,w\in\F_q$ and precisely one of $u,v,w$ being zero, where $q$ is even.
\end{enumerate}
\end{thm}

\begin{thm} \label{2}
The points in\/ $\P^2(\F_{q^2})$ on the curve $u^{q-1} + v^{q-1} + w^{q-1} = 0$ are as follows:
\begin{enumerate}
\item points $(c v^2 : v : 1)$ with $c\in\F_q^*$ and $v^{q-1}$ a primitive cube root of unity, where $q\equiv 2\pmod{3}$;
\item points $(u : v : 1)$ with $u,v\in\F_q^*$, where $q\equiv 0\pmod{3}$;
\item points $(u : v : w)$ with $\{u,v,w\}=\{0,1,d\}$ where $d^{q-1}=-1$.
\end{enumerate}
\end{thm}

The description of the points over $\F_q$ is an easy exercise, and we record the answer just for completeness:

\begin{lemma} \label{lemma}
The points in\/ $\P^2(\F_q)$ on the curve $u^{q-1} + v^{q-1} + w^{q-1} = 0$ are as follows:
\begin{enumerate}
\item points $(u : v : w)$ with $u,v,w\in\F_q^*$, where $p=3$; and
\item points $(u : v : w)$ with $u,v,w\in\F_q$ and exactly one of $u,v,w$ being zero, where $p=2$.
\end{enumerate}
\end{lemma}

These results yield formulas for the numbers of $\F_{q^i}$-rational points on these curves, which agree with the
corresponding formulas in \cite{Moisio}.  Our proofs are much shorter and more direct than those in \cite{Moisio}, but it should
be noted that
the lengthier proofs in \cite{Moisio} also yielded results about certain twists of these Fermat curves, which we do not
consider here.

Our technique also yields a description of the $\F_{q^2}$-rational points on the degree-$(q-1)$ Fermat surface:

\begin{thm} \label{surface}
The points in\/ $\P^3(\F_{q^2})$ on the surface $u^{q-1}+v^{q-1}+w^{q-1}+x^{q-1}=0$ are as follows:
\begin{enumerate}
\item points $(u_1 : u_2 : u_3 : u_4)$ with $u_i\in\F_{q^2}$, where at least one $u_i$ is nonzero and there is a permutation $\sigma$ of $\{1,2,3,4\}$ such
that $u_{\sigma(1)}^{q-1}=-u_{\sigma(2)}^{q-1}$ and $u_{\sigma(3)}^{q-1}=-u_{\sigma(4)}^{q-1}$;
\item points $(u_1 : u_2 : u_3 : u_4)$ in which exactly one $u_i$ is zero and the $(q-1)$-th powers of the other $u_i$'s are distinct cube roots of
unity, where $q\equiv 2\pmod{3}$;
\item points $(u_1 : u_2 : u_3 : u_4)$ with $u_i\in\F_q$ and exactly one $u_i=0$, where $q\equiv 0\pmod{3}$.
\end{enumerate}
\end{thm}

The points in (1) above are the points which are on the lines contained in the surface.

This theorem has the following consequences:

\begin{cor} \label{surfaceprod}
If $u,v,w,x\in\F_{q^2}$ satisfy $u^{q-1}+v^{q-1}+w^{q-1}+x^{q-1}=0$, then $uvwx$ is a square in $\F_{q^2}$.
\end{cor}

\begin{cor} \label{surfacecount}
The number $N$ of points in\/ $\P^3(\F_{q^2})$ on the surface $u^{q-1}+v^{q-1}+w^{q-1}+x^{q-1}=0$ satisfies
\[
N =
\begin{cases}
3(q-1)^4 + 3(q-1)^3 + 6(q-1) & \text{ if $q\equiv 1\pmod{6}$} \\
3(q-1)^4 + 3(q-1)^3 + 8(q-1)^2 + 6(q-1) & \text{ if $q\equiv 5\pmod{6}$} \\
3(q-1)^4 + 3(q-1)^3 + 4(q-1)^2 + 6(q-1) & \text{ if $q\equiv 0\pmod{3}$} \\
(3q+1)(q-1)^3 + 6(q-1) & \text{ if $q\equiv 4\pmod{6}$} \\
(3q+1)(q-1)^3 + 8(q-1)^2 + 6(q-1) & \text{ if $q\equiv 2\pmod{6}$.}
\end{cases}
\]
\end{cor}

Our proof of Theorem~\ref{3} relies on the unexpected factorization (\ref{unexp}).  A similar type of unexpected factorization
arose in a paper of Carlitz \cite[eqn.\ (25)]{Carlitz} on multiple Kloosterman sums.  Carlitz's result was then used in \cite{Moisio},
together with a few pages of additional character sum computations, to count the number of $\F_{q^3}$-rational points on
the degree-$(q-1)$ Fermat curve.  It is tempting to surmise that there should
be a natural bijection which shows that our factorization and Carlitz's are in some sense the same factorization in
different languages.  However, we have not been able to find such a bijection.  Likewise, our proof of Theorem~\ref{surface} relies on the unexpected factorization (\ref{unexp2}).

We will prove Theorems~\ref{2} and \ref{3} in the next two sections, and prove Corollary~\ref{cor} in Section 5.
We conclude in Section 6 with proofs of Theorem~\ref{surface} and its corollaries.


\section{Points over $\F_{q^2}$}

In this section we prove Theorem~\ref{2}, and then verify that the resulting formula for the number of $\F_{q^2}$-rational points
agrees with the previously known value.

\begin{proof}[Proof of Theorem~\ref{2}]
Suppose $u,v\in\F_{q^2}^*$ satisfy $u^{q-1} + v^{q-1} + 1 = 0$.  Then
\begin{align*}
1 &= (-u^{q-1})^{q+1} \\
&= (v^{q-1}+1)^{q+1} \\
&= (v^{q-1}+1)^q (v^{q-1}+1) \\
&= (v^{q^2-q}+1)(v^{q-1}+1) \\
&= v^{q^2-1} + v^{q^2-q} + v^{q-1} + 1 \\
&= 1 + v^{1-q} + v^{q-1} + 1,
\end{align*}
or equivalently,
\[
0 = v^{2q-2} + v^{q-1} + 1.
\]
Thus, if $p\ne 3$ then $v^{q-1}$ is a primitive cube root of unity, and if
$p=3$ then $v^{q-1}=1$.

Conversely, if $p\ne 3$ then $\F_{q^2}$ contains two primitive cube roots of unity.
In order that these cube roots of unity should be $(q-1)$-th powers, it is necessary
and sufficient that $3(q-1)\mid (q^2-1)$, or equivalently, $q\equiv 2\pmod{3}$.
If $q\equiv 2\pmod{3}$ then there are precisely $2q-2$ elements $v\in\F_{q^2}^*$ such that
$v^{q-1}$ is a primitive cube root of unity, and for any such $v$ the equation $u^{q-1}+v^{q-1}+1=0$
can be rewritten as $u^{q-1} = v^{2q-2}$, or equivalently, $u=cv^2$ with $c\in\F_q^*$.
Likewise, if $p=3$ and $v\in\F_q^*$ then the equation $u^{q-1}+v^{q-1}+1=0$ becomes $u^{q-1}=1$, so $u\in\F_q^*$.
Finally, it is straightforward to solve $u^{q-1}+v^{q-1}+w^{q-1}=0$ when $uvw=0$.  This completes the proof.
\end{proof}

The following consequence of Theorem~\ref{2} is immediate.

\begin{cor} \label{2cor}
The number $N$ of points in\/ $\P^2(\F_{q^2})$ on the curve $u^{q-1}+v^{q-1}+w^{q-1}=0$ is
\[
N = \begin{cases}
3(q-1) & \text{ if $q\equiv 1\pmod{3}$} \\
3(q-1)+(q-1)^2 & \text{ if $q\equiv 0\pmod{3}$} \\
3(q-1)+2(q-1)^2 & \text{ if $q\equiv 2\pmod{3}$.}
\end{cases}
\]
\end{cor}


\section{Points over $\F_{q^3}$}

In this section we prove Theorem~\ref{3}, and then verify that the resulting formula for the number of $\F_{q^3}$-rational points
agrees with the previously known value.

\begin{proof}[Proof of Theorem~\ref{3}]
Suppose  $u\in\F_{q^3}^*$ and $v\in\F_{q^3}$  satisfy  $u^{q-1} + v^{q-1} + 1 = 0$.  Then
\[
-1 = (-u^{q-1})^{q^2+q+1} = (v^{q-1} + 1)^{q^2+q+1}.
\]
If $v=0$ then we get $-1=1$, so $q$ is even and $u^{q-1}=1$.  Henceforth assume $v\ne 0$.
Write  $V := v^{q-1}$, so that
\begin{align*}
-1 &= (V+1)^{q^2}(V+1)^q(V+1) \\
   &= (V^{q^2}+1)(V^q+1)(V+1) \\
   &= V^{q^2+q+1} + V^{q^2+q} + V^{q^2+1} + V^{q^2} + V^{q+1} + V^q + V + 1.
\end{align*}
Upon substituting  $V^{q^2} = 1/V^{q+1}$,  we obtain
\[
-1 = 1 + V^{-1} + V^{-q} + V^{-q-1} + V^{q+1} + V^q + V + 1.
\]
Adding $1$ to both sides yields
\begin{equation} \label{unexp}
0 = (V^{q+1} + V + 1) \cdot (V^{-q-1} + V^{-1} + 1).
\end{equation}
This is the ``unexpected factorization" discussed at the end of Section~1.
Since $V = v^{q-1}$, this last equation says
\[
0 = (v^{q^2-1} + v^{q-1} + 1)\cdot ( (1/v)^{q^2-1} + (1/v)^{q-1} + 1),\]
or equivalently,
\[
0 = T(v) \cdot T(1/v).
\]

Conversely, if $T(v)=0$ and $v\ne 0$ then $v\in\F_{q^3}$, since
\[
0 = 0^q - 0 = T(v)^q - T(v) = v^{q^3} - v.
\]
In this case, $v^{q-1} + 1 = -v^{q^2-1}$, so for $u\in\F_{q^3}^*$ the condition $u^{q-1}+v^{q-1}+1=0$
says that $u^{q-1} = v^{q^2-1}$, or equivalently, $u=cv^{q+1}$ for some $c\in\F_q^*$.
Likewise, if $v\ne 0$ and $T(1/v)=0$ then $v\in\F_{q^3}$ and the condition $u^{q-1}+v^{q-1}+1=0$ becomes
$(u/v)^{q-1} + 1 + (1/v)^{q-1}=0$, or equivalently $(u/v)^{q-1} = (1/v)^{q^2-1}$.  This last equation says
that $u/v = c/v^{q+1}$ for some $c\in\F_q^*$, so that $u=c/v^q$.
Finally, if $q$ is even and $u\in\F_q^*$ then plainly $u^{q-1}+1=0$.  This completes the proof.
\end{proof}

\begin{cor} \label{3cor}
The number $N$  of points in\/ $\P^2(\F_{q^3})$ on the curve $u^{q-1}+v^{q-1}+w^{q-1}=0$ is
\[
N = \begin{cases}
(2q+2)(q-1)^2 & \text{ if $q\equiv 5\pmod{6}$} \\
(2q+1)(q-1)^2 & \text{ if $q\equiv 0\pmod{3}$} \\
2q(q-1)^2 & \text{ if $q\equiv 1\pmod{6}$} \\
(2q+2)(q-1)^2+3(q-1) & \text{ if $q\equiv 2\pmod{6}$} \\
2q(q-1)^2 + 3(q-1) & \text{ if $q\equiv 4\pmod{6}$.}
\end{cases}
\]
\end{cor}

\begin{proof}
If $q$ is odd then all points have nonzero coordinates; if $q$ is even then there are precisely
$3(q-1)$ points with a zero coordinate.  Henceforth we consider points with nonzero coordinates.
Since the derivative $T'(X)$ is a nonzero constant, we know that $T(X)$ is squarefree, and hence
has $q^2-1$ distinct nonzero roots.  If none of these are roots of $T(1/X)$ then the number
of $\F_{q^3}$-rational points on our curve which have nonzero coordinates will equal
$2(q^2-1)(q-1)$, or in other words $(2q+2)(q-1)^2$.  Next, if $T(v)=T(1/v)=0$ then
\[
v^{q^2}+v^q+v=0=v^{q^2+1}T(1/v)= v+v^{q^2-q+1}+v^{q^2},
\]
so $v^q=v^{q^2-q+1}$ and thus $v^{q^2-2q+1}=1$.  Since $v\in\F_{q^3}$, it follows
that $v^{\gcd(q^2-2q+1,q^3-1)}=1$.  We compute $q^3-1=(q^2-2q+1)(q+2)+3q-3$,
so $\gcd(q^2-2q+1,q^3-1)$ equals $3q-3$ if $q\equiv 1\pmod{3}$, and equals $q-1$ otherwise.
Conversely, if $v\in\F_q^*$ then $T(v)=3v$ is nonzero unless $q\equiv 0\pmod{3}$,
in which case $T(v)=T(1/v)=0$ and
\[
\{(c v^{q+1} : v : 1) \colon c\in\F_q^*\} = \{ (c v^{-q} : v : 1) \colon c\in\F_q^*\}.
\]
Finally, if $q\equiv 1\pmod{3}$ and $\omega:=v^{q-1}$ has order $3$, then
$T(v)=v(\omega^{q+1}+\omega+1)=0$, and likewise $T(1/v)=0$.  In this case,
$v^3$ is in $\F_q^*$, so also $v^{2q+1}\in\F_q^*$, which implies that again
\[
\{(c v^{q+1} : v : 1) \colon c\in\F_q^*\} = \{ (c v^{-q} : v : 1) \colon c\in\F_q^*\}.
\]
The result now follows from Theorem~\ref{3}.
\end{proof}


\section{The product of the coordinates}

We now prove Corollary~\ref{cor}.  The conclusion clearly holds if $uvw=0$, so assume that $uvw\ne 0$.
Without loss, we may divide each of $u,v,w$ by $w$, in order to assume that $w=1$.
The conclusion is immediate if every element of $\F_{q^i}$ is a cube, which occurs when either $p=3$
or both $i=3$ and $q\equiv 2\pmod{3}$.  Henceforth assume that neither of these situations holds.
This rules out all solutions if $i=1$.
If $i=2$ then $q\equiv 2\pmod{3}$ and $u=cv^2$ with $c\in\F_q^*$, so $uv=cv^3$ is a cube since $c$ is a $(q+1)$-th power.
If $i=3$ then $q\equiv 1\pmod{3}$ and, for some $c\in\F_q^*$, either $u=c v^{q+1}$ or $u=c v^{-q}$.
Thus $uvw$ equals either $c v^{q+2}$ or $c v^{1-q}$.  Since $q\equiv 1\pmod{3}$, both $v^{q+2}$ and $v^{1-q}$
are cubes in $\F_{q^3}$.  Finally, since $c$ is in $\F_q^*$, it is a $(q^2+q+1)$-th power in $\F_{q^3}$, and
hence is a cube.  The result follows.

\begin{remark}
Since the statement of Corollary~\ref{cor} is quite clean, it is natural to wonder whether there is a generalization
to higher-degree extensions of $\F_q$.  However, the most immediate generalization is not true.  For instance,
there are elements $u,v,w\in\F_{16}$ for which $u+v+w=0$ but $uvw$ generates the group $\F_{16}^*$.
This shows that the immediate generalization of Corollary~\ref{cor} does not hold for $i=4$ and $q=2$.
\end{remark}


\section{The Fermat surface}

In this section we prove Theorem~\ref{surface} and its corollaries.

\begin{proof}[Proof of Theorem~\ref{surface}]
Let $u,v,w,x\in\F_{q^2}$ satisfy $u^{q-1}+v^{q-1}+w^{q-1}+x^{q-1}=0$.
The desired conclusion follows from Theorem~\ref{2} if $uvwx=0$, so assume that $u,v,w,x\in\F_{q^2}^*$.  Then
\begin{align*}
1 &= (-x^{q-1})^{q+1} \\
&= (u^{q-1}+v^{q-1}+w^{q-1})^{q+1} \\
&= (u^{q-1}+v^{q-1}+w^{q-1})^q \cdot (u^{q-1}+v^{q-1}+w^{q-1}) \\
&= (u^{q^2-q}+v^{q^2-q}+w^{q^2-q}) \cdot (u^{q-1}+v^{q-1}+w^{q-1}) \\
&= (u^{1-q}+v^{1-q}+w^{1-q}) \cdot (u^{q-1}+v^{q-1}+w^{q-1}) \\
&= 3 + (u/v)^{q-1}+(v/u)^{q-1}+(u/w)^{q-1}+(w/u)^{q-1}+(v/w)^{q-1}+(w/v)^{q-1}.
\end{align*}
After subtracting $1$ from both sides, and then multiplying both sides by $(uvw)^{q-1}$, we obtain
\begin{equation} \label{unexp2}
0=(u^{q-1}+v^{q-1})(u^{q-1}+w^{q-1})(v^{q-1}+w^{q-1}).
\end{equation}
Hence two of $u^{q-1}$, $v^{q-1}$, and $w^{q-1}$ are negatives of one another.  The result follows.
\end{proof}

Corollary~\ref{surfaceprod} follows from Theorem~\ref{surface} because, if $u,v,w,x\in\F_{q^2}^*$ satisfy
$u^{q-1}=-v^{q-1}$ and $w^{q-1}=-x^{q-1}$, then $(uvwx)^{q-1}=(vx)^{2q-2}$ is a square.
Finally, the deduction of Corollary~\ref{surfacecount} from Theorem~\ref{surface} is straightforward.

\end{document}